\documentclass{amsart}
\usepackage{amstext}
\usepackage{amsthm}
\usepackage{amsopn}
\usepackage{amsfonts}
\usepackage{amsmath}
\usepackage{amssymb}
\usepackage{amsthm}
\usepackage{amscd}

\newcommand{\norm}[2]{\left\| {#1}\right\| _{#2}}
\newcommand{\abs}[1]{\left | {#1}\right | }
\newcommand{\all}[2]{ \left \{\, {#1} \, : \, {#2} \, \right \} }
\newcommand{\set}[1]{{\{#1\}}}
\newcommand{\qn}{quasi-nilpotent}

\newcommand{\norml}[1]{\norm{{#1}}{}}

\newcommand{\Si}{S_\infty}
\newcommand{\wh}{\widehat}
\newcommand{\cal}{\mathcal}

\newcommand{\N}{\mathbb{N}}
\newcommand{\Z}{\mathbb{Z}}
\newcommand{\R}{\mathbb{R}}
\newcommand{\C}{\mathbb{C}}
\newcommand{\lam}{\lambda}

\theoremstyle{plain}
\newtheorem{theorem}{Theorem}[section]
\newtheorem{proposition}[theorem]{Proposition}

\newtheorem{corollary}[theorem]{Corollary}

\theoremstyle{definition}
\newtheorem{definition}[theorem]{Definition}
\newtheorem{remark}[theorem]{Remark}
\newtheorem{notation}[theorem]{Notation}
\newtheorem{example}[theorem]{Example}

\newtheoremstyle{remarkbreak}
  {9pt}
  {9pt}
  {\upshape}
  {}
  {\bfseries}
  {.}
  {\newline}
  {}

\theoremstyle{remarkbreak}
\newtheorem{rembreak}[theorem]{Remark}

\begin{document}

\title[Resolvent estimates]
{Resolvent estimates for operators belonging to exponential classes}

\author[O.F.~Bandtlow]{Oscar F.~Bandtlow}
\address{
School of Mathematical Sciences\\
Queen Mary, University of London\\
London E3 4NS, UK}
\email{o.bandtlow@qmul.ac.uk}

\date{14 July 2008}

\begin{abstract}
For $a,\alpha>0$ let $E(a,\alpha)$ be the set of all compact operators $A$ on a
separable 
Hilbert space such that $s_n(A)=O(\exp( -an^\alpha))$, where $s_n(A)$
denotes the $n$-th singular number of $A$. We provide upper bounds for
the norm of the resolvent $(zI-A)^{-1}$ of $A$ in terms of a quantity describing the
departure from normality of $A$ and the distance of $z$ to the
spectrum of $A$. As a consequence we obtain upper bounds for the
Hausdorff distance of the spectra of two operators in $E(a,\alpha)$.  

\end{abstract}

\maketitle

\section{Introduction} 
Let $A$ and $B$ be compact operators on a Hilbert space. It is known that
if $\norm{A-B}{}$ is small then the spectra of $A$ and $B$ are
close in a suitable sense (for example, with respect to the Hausdorff
metric on the space of compact subsets of $\C$). Just how close are they?
Standard perturbation theory gives bounds in terms of quantities that
require a rather detailed knowledge of the spectral properties of both
operators, for example the norms of the resolvents of $A$ and $B$
along contours in the complex plane, which are difficult to
obtain in practice. 

The main concern of this article is to derive an upper bound for the
norm of the resolvent $(zI-A)^{-1}$ of an operator $A$ belonging 
to certain subclasses of compact operators in terms of 
simple, readily computable quantities, typically 
involving the distance of $z$ to the spectrum of $A$ and a number
measuring the departure from normality of $A$. 
As a result, we obtain simple upper bounds for 
the Hausdorff distance of the spectra of two operators in these
subclasses. Estimates of this type have previously been obtained for
operators in the Schatten classes (see \cite{Gil,Ban}) and more
generally (but less sharp), for
operators belonging to $\Phi$-ideals (see \cite{Pok}). 

These subclasses, termed \textit{exponential classes}, 
are constructed as follows. For $a, \alpha>0$ let
$E(a,\alpha)$ denote the collection 
of all compact operators on a separable Hilbert space for
which $s_n(A)=O(\exp(-an^{\alpha}))$, where $s_n(A)$ denotes the $n$-th
singular number of $A$. As we shall see, $E(a,\alpha)$ is not a linear
space (see Remark~\ref{rem2}), hence {\it a fortiori} not an operator
ideal, 
and may thus be
viewed as a slightly pathological object in this context. 
There is nevertheless compelling 
reason to consider these classes: on the one hand, 
the resolvent bounds given in \cite{Ban,Gil,Pok},
while applicable to operators in $E(a,\alpha)$, can be improved
significantly (see Remark~\ref{rem:improvement}), 
the lack of linear structure posing
almost no problem for the derivation of these improvements. 
On the other hand, operators belonging to exponential classes 
arise naturally in a
number of different ways. 
For example, if $A$ is an integral operator with real
analytic kernel given as a function on $[0,1]^d\times [0,1]^d$, then $A\in
E(a,1/d)$ for some $a>0$ (see \cite{KR}). Other examples 
of operators in the exponential class $E(a,1/d)$ for some $a>0$
include composition operators
on Bergman spaces over domains in $\C^d$ whose symbols are strict
contractions, or more generally  
transfer operators corresponding to holomorphic map-weight systems on
$\C^d$, the latter providing one of the motivations to look
more closely into the properties of operators belonging to exponential
classes (see Example~\ref{expoclassexamples}~(iv) 
and \cite{rates,ruelle,decay}). 

This article is organised as follows. In
Section~\ref{section:expoclass} we define the exponential
classes and study some of their properties. In particular, we shall
give a sharp description of the behaviour of exponential classes
under addition (see Proposition~\ref{propaddition}), 
and a sharp characterisation of the
eigenvalue asymptotics of an operator in a given exponential class (see
Proposition~\ref{prop:weylest}). 
Some of the arguments in
this section rely on results concerning monotonic arrangements of
sequences, which are presented in the Appendix. 
In Section~\ref{section:resest} we will use techniques similar to those
already 
employed in
\cite{Ban} to obtain resolvent estimates for operators in
$E(a,\alpha)$ (see Theorem~\ref{carle:theo}). 
In particular we shall give a sharp estimate for the
growth of the resolvent of a \qn\ operator in $E(a,\alpha)$ (see
Proposition~\ref{prop3}). 
These estimates will then be used in the final section to deduce
Theorem~\ref{theo:specdist}, which provides 
spectral variation and spectral distance formulae for operators in 
$E(a,\alpha)$.

\begin{notation}
Throughout this article $H$ and $H_i$ will be assumed to be 
separable Hilbert spaces.
We use $L(H_1,H_2)$ to denote
the Banach space of bounded linear
operators from $H_1$ to $H_2$ equipped with the usual norm 
and $\Si(H_1,H_2)\subset L(H_1,H_2)$
to denote the closed subspace of compact operators from $H_1$ to $H_2$.
We shall often write $L$ or $\Si$ if the Hilbert spaces $H_1$ and
$H_2$ are understood. 

For $A\in
\Si(H_1,H_2)$ we use
\[ s_k(A):= \inf \all{\norml{A-F}}{F\in L(H_1,H_2),\, 
{\rm rank}(F)<k}\quad (k\in\N)\]
to denote the \emph{$k$-th approximation number}
of $A$, and $s(A)$ to denote the sequence $\{s_k(A)\}_{k=1}^\infty$.

The spectrum and the resolvent set of $A\in L(H,H)$ will be denoted by 
$\sigma(A)$ and $\varrho(A)$, respectively. 
For $A\in\Si(H,H)$ we let $\lambda(A)=\set{\lam_k(A)}_{k=1}^\infty$ 
denote the
sequence of eigenvalues of $A$, 
each eigenvalue repeated according to
its algebraic multiplicity, and
ordered by magnitude, so that
$\abs{\lam_1(A)}\geq\abs{\lam_2(A)}\geq\ldots$. Similarly, we write
$|\lambda(A)|$ for the sequence $\set{|\lam_k(A)|}_{k=1}^\infty$. 

We note that the approximation numbers coincide with the
\emph{singular numbers}, that is,
$$
s_k(A)= \sqrt{\lambda_k(A^*A)}\quad (k\in\N)\,,
$$
where $A^*\in L(H_2,H_1)$ denotes the adjoint of $A\in L(H_1,H_2)$. 
For more
information about these notions see, for
example, \cite{Pie,GK,DS2, Rin}. 
\end{notation}

\section{Exponential classes}
\label{section:expoclass}
Exponential classes arise by grouping together all 
operators $A$ whose singular numbers $s_n(A)$ 
decay at a given (stretched) exponential rate, that is,
$s_n(A)=O(\exp(-an^\alpha))$ for fixed $a>0$ and $\alpha>0$.  
Our main concern in this section will be to
investigate how these classes behave under addition and
multiplication, and to determine the rate of decay of the eigenvalue
sequence of an operator in a given class. Some of the arguments in
this section depend on results concerning monotonic arrangements of 
sequences, which are discussed in the Appendix. 

\begin{definition}
  Let $a>0$ and $\alpha>0$. Then 
 \[  {\cal E} (a,\alpha):=\all{x\in \C^\N}{ |x|_{a,\alpha}:=\sup_{n\in\N}
  |x_n| \exp(an^{\alpha}) < \infty}, \]
and 
\[ E(a,\alpha;H_1,H_2):=\all{A\in\Si(H_1,H_2)}{ |A|_{a,\alpha}:=\sup_{n\in\N}
  s_n(A)\exp(an^{\alpha}) < \infty}. \]
are called {\it exponential classes of type $(a,\alpha)$ of sequences}
and {\it operators}, respectively. The numbers
$|x|_{a,\alpha}$ and  
$|A|_{a,\alpha}$ are called {\it $(a,\alpha)$-gauge} or simply {\it
  gauge} of $x$ and $A$, respectively.  
\end{definition}
Whenever the Hilbert spaces are clear from the context, we suppress
reference to them and simply write $E(a,\alpha)$ instead of
$E(a,\alpha;H_1,H_2)$. 

\begin{remark}
Note that ${\cal E}(a,\alpha)$ is a Banach space when equipped with the
gauge $|\cdot |_{a,\alpha}$. On the other hand, the set $E(a,\alpha)$, 
the non-commutative
analogue of ${\cal E}(a,\alpha)$, is not
even a linear space in general 
(see Proposition~\ref{propaddition} and Remark~\ref{rem2}
below). 
  The
  reason for this is that if a sequence lies in ${\cal E}(a,\alpha)$ then a
  rearrangement of this sequence need not; in particular ${\cal E}(a,\alpha)$
  is not a Calkin space in the sense of \cite[p.~26]{simon} (cf.~also
  \cite{calkin}).  However, $E(a,\alpha;H,H)$ turns out to be a 
\textit{pre-ideal} (see
  Remark~\ref{preidealrem}).
\end{remark}

Operators belonging to exponential classes arise naturally in a number
of different contexts. 
\begin{example}\label{expoclassexamples}\mbox{}

(i) Let $\sigma$ be a complex measure on the circle group $\cal T$
such that its Fourier transform satisfies $$|\wh \sigma (n)| \leq
\exp(-a|n|) \qquad (n \in \Z)\,.$$ 
It is not difficult to see that this is the case if and only if
$\sigma$ is absolutely continuous
with respect to Haar measure on $\cal T$ and the corresponding
Radon-Nikod\'{y}m derivative is holomorphic on $\cal T$.

Let $ L^2({\cal T})$ be the
complex Hilbert space of square-integrable functions on $\cal T$,
with respect to Haar measure on $\cal T$. Let $A: L^2({\cal
T}) \longrightarrow L^2({\cal T})$ be the convolution operator
$$Af= f * \sigma\,.$$
The spectrum of $A$ is $\wh \sigma(\Z) \cup\{0\}$ and the spectrum
of $A^*A$ equals $\wh{ \sigma * \widetilde\sigma}(\Z) \cup\{0\} =
|\wh \sigma(\Z)| \cup\{0\}$, where $d\widetilde\sigma(t) = d\sigma
(t^{-1})$ (cf. \cite[p.~87]{b}). Moreover, $A$ is
a compact operator and the non-zero eigenvalues of $A$ are precisely
the numbers $\wh \sigma (n)$ for $n\in\Z$. In order to locate $A$ in the
scale of exponential classes, we enumerate the eigenvalues of $A$ as follows
$$x_n
= \wh \sigma\left(\frac{(-1)^n}{4} (2n +(-1)^n -1)\right)\qquad (n
\in \N).$$ 
Then the sequence $x$ belongs to the class ${\cal
E}(a/2,1)$ with $|x|_{1/2,1}\leq \exp(a/2)$ since
\begin{eqnarray*}
& & \left|\widehat \sigma\left(\frac{(-1)^n}{4} (2n +(-1)^n
-1)\right)\right|\\& \leq & \exp\left(-a\left|\frac{(-1)^n}{4} (2n
+(-1)^n -1)\right|\right)\\ & \leq&
\exp\left(-\frac{a}{2}(n-1)\right) =
\exp\left(\frac{a}{2}\right)\exp\left(-\frac{a}{2}n\right).
\end{eqnarray*}
By Corollary~\ref{app:arrangementcoro} 
the decreasing arrangement $x^{(+)}$ of $x$ also
belongs to ${\cal E}(a/2,1)$ with $|x^{(+)}|_{1/2,1}\leq \exp(a/2)$. 
Thus $s(A)=|\lambda(A)|\in {\cal E}(a/2,1)$ and it follows that  
$A\in E(a/2,1)$ with $|A|_{a/2,1}\leq \exp(a/2)$. 

(ii) A variant of the above example is discussed by K\"onig 
and Richter~\cite{KR}, who showed that 
if $A$ is an integral operator on the space of Lebesgue
  square-integrable functions on the $d$-dimensional unit-cube
  $[0,1]^d$ whose kernel is real analytic on $[0,1]^d\times [0,1]^d$,
  then $A\in E(1/d)$. 

(iii) For a domain $\Omega\subset \R^d$ with $d>1$ 
let $h^2(\Omega)$ be the Bergman space
of Lebesgue square-integrable harmonic functions on $\Omega$. If
$\Omega_1,\Omega_2\subset \R^d$ are two domains such that $\Omega_2$
is \textit{compactly contained} in $\Omega_1$, that is
$\overline{\Omega_2}$ is a compact subset of $\Omega_1$, then the
canonical embedding $J:h^2(\Omega_1)\hookrightarrow h^2(\Omega_2)$
given by $Jf=f|\Omega_2$ belongs to the exponential class
$E(1/(d-1))$. Moreover, for domains $\Omega_1$ and $\Omega_2$
with simple geometries it is possible to sharply locate $J$ in an
exponential class $E(a,1/(d-1))$ and calculate the
corresponding $(a,1/(d-1))$-gauge of $J$ exactly. See \cite{BC}.  

(iv) For $\Omega\subset \C^d$ a bounded domain, let $L^2_{Hol}(\Omega)$
denote the Bergman space of holomorphic functions which are
square-integrable with respect to $2d$-dimensional Lebesgue measure on
$\Omega$. Given a collection $\set{\phi_1,\ldots,\phi_K}$ of
holomorphic maps $\phi_k:\Omega\to\Omega$ and a collection
$\set{w_1,\ldots,w_K}$ of bounded holomorphic functions
$w_k:\Omega\to\C$ consider the corresponding linear 
operator $A$ on $L^2_{Hol}(\Omega)$ given by
\[ Af:=\sum_{k=1}^Kw_k\cdot f\circ\phi_k\,.\]
If $\cup_{k}\phi_k(\Omega)$ is compactly contained in $\Omega$ (see
the previous example for the definition) then $A$ is a compact
endomorphism of $L^2_{Hol}(\Omega)$ and $A \in E(a,1/d)$, where $a$ depends
on the geometry of $\Omega$ and $\cup_{k}\phi_k(\Omega)$ (see
\cite{decay}). 

Operators of this type, known as {\it transfer operators}, 
play an important role in the ergodic theory of
expanding dynamical systems due to the remarkable fact 
that their spectral data can be used to gain   
insight into geometric and dynamic invariants of a given expanding 
dynamical system (see \cite{ruellebook}). As a consequence, it is of
interest to determine spectral properties of these operators exactly,
or at least to a given accuracy. The latter problem, namely that of
calculating rigorous error
bounds for spectral approximation procedures for these operators
provided one of the main motivations to study operators in exponential
classes (see  \cite{decay}). 
\end{example}

We shall now study some of the properties of the classes
$E(a,\alpha)$. First we note that if we order the indices $(a,\alpha)$
reverse lexicographically, that is, by defining 
\[ (a,\alpha)\prec (a',\alpha'):\Leftrightarrow (\alpha<\alpha') \text{
  or } (\alpha=\alpha' \text{ and } a<a'), \]
then we obtain the following inclusions.
\begin{proposition} Let $a,a'>0$ and $\alpha,\alpha'>0$. Then
\begin{itemize}
\item[(i)] $(a,\alpha)\prec (a',\alpha') \Leftrightarrow
  {\cal E}(a',\alpha')\subsetneqq {\cal E}(a,\alpha)$;
\item[(ii)] $(a,\alpha)\prec (a',\alpha') \Leftrightarrow
  E(a',\alpha')\subsetneqq E(a,\alpha)$.
\end{itemize}
\end{proposition}
\begin{proof}
  The proof of (i) is straightforward and will be omitted. 
Assertion (ii) follows from (i) together with the observation that $A\in
E(a,\alpha)$ iff $s(A)\in{\cal E}(a,\alpha)$ and the fact that
for every monotonically decreasing 
$x\in{\cal E}(a,\alpha)$ there is a compact $A$ with $s(A)=x$. 

\end{proof}

While $E(a,\alpha)$ is not, in general, a
linear space, it does enjoy 
the following closure
properties.

\begin{proposition}\label{prop1} Let $a,\alpha>0$. 
If $A\in L(H_2,H_1)$, $B\in E(a,\alpha;H_3,H_2)$, and $C\in
L(H_4,H_3)$, 
then $\abs{ABC}_{a,\alpha}\leq \norml{A}\abs{B}_{a,\alpha}\norml{C}$. 
In particular,
  \[ L(H_2,H_1)\,E(a,\alpha;H_3,H_2)\,L(H_4,H_2)\subset 
    E(a,\alpha;H_4,H_1). \]
\end{proposition}
\begin{proof}  Follows from
\[ s_k(ABC)\leq\norml{A}s_k(B)\norml{C} \]
for $k\in\N$ (see \cite[2.2]{Pie}).
\end{proof}

\begin{remark}
\label{preidealrem}
  The proposition implies that 
\[ L(H,H)\,E(a,\alpha;H,H)\,L(H,H)\subset E(a,\alpha;H,H)\,. \]
Thus the classes $E(a,\alpha;H,H)$, while lacking linear structure,
satisfy part of the definition of an operator ideal. In other words,
$E(a,\alpha;H,H)$ is what is sometimes referred to as a pre-ideal (see,
for example, \cite{Nel}). 
\end{remark}

We now consider in more detail the relation between different
exponential classes under addition. We start with a general result
concerning the singular numbers of a sum of operators.  

\begin{proposition}\label{berg:prop10}
Let $A_k\in \Si(H_1,H_2)$ for $1\leq k\leq K$.  Then
\[ s_n\left(\sum_{k=1}^KA_k\right)
\leq K\sigma_n \quad (n\in\N), \] where $\sigma$ denotes the
decreasing arrangement (see the Appendix) 
of the $K$ singular number sequences
$s(A_1),\ldots,s(A_K)$.

\end{proposition}
\begin{proof}
  Set $A:=\sum_{k=1}^KA_k$. The compactness of the $A_k$ means they
  have Schmidt representations
\[ A_k=\sum_{l=1}^\infty s_l(A_k)a_l^{(k)}\otimes b_l^{(k)}, \]
where $\set{a_l^{(k)}}_{l\in\N}$ and $\set{b_l^{(k)}}_{l\in\N}$ are
suitable orthonormal systems in $H_1$ and $H_2$ respectively. Here 
$a\otimes b$, where $a\in H_1$ and $b\in H_2$, denotes the rank-1
operator 
$H_1\to H_2$ given by 
\[ (a \otimes b)x=(x,a)_{H_1}\, b\,. \]

Let $\nu:\N\rightarrow \N$ and $\kappa:\N\rightarrow \N$ be functions
that effect the decreasing arrangement of the singular numbers of
the $A_k$ in the sense that
\[ \sigma_n=s_{\nu(n)}(A_{\kappa(n)}) \quad (n\in\N). \]
Then 
\[ A=\sum_{n=1}^\infty \sigma_n a_{\nu(n)}^{(\kappa(n))}\otimes
  b_{\nu(n)}^{(\kappa(n))}\,,\]
which suggests defining, for each $m\in\N_0$, the rank-$m$ operator 
$F_m:H_1\rightarrow H_2$ by
\begin{align*}
  F_0&:=0\,, \\
  F_m&:=\sum_{n=1}^m\sigma_n a_{\nu(n)}^{(\kappa(n))}\otimes
  b_{\nu(n)}^{(\kappa(n))} \quad (m\in\N).
\end{align*}
If $x\in H_1$ and $y\in H_2$ then
\begin{align*}
  \abs{ ((A-F_{m-1})x,y)_{H_2}}& \leq
  \sum_{n=m}^\infty\sigma_n\abs{(x,a_{\nu(n)}^{(\kappa(n))})_{H_1}
    (b_{\nu(n)}^{(\kappa(n))},y)_{H_2}}\\
&\leq \sigma_m \sum_{n=1}^\infty\abs{(x,a_{\nu(n)}^{(\kappa(n))})_{H_1}%
  (b_{\nu(n)}^{(\kappa(n))},y)_{H_2}}\\
&=\sigma_m\sum_{k=1}^K\sum_{l=1}^\infty\abs{(x,a_{l}^{(k)})_{H_1}%
  (b_{l}^{(k)},y)_{H_2}}\\
&\leq \sigma_m\sum_{k=1}^K 
  \sqrt{ \sum_{l=1}^\infty \abs{(x,a_l^{(k)})_{H_1}}^2}%
  \sqrt{ \sum_{l=1}^\infty \abs{(b_l^{(k)},y)_{H_2}}^2}\\
&\leq \sigma_m \sum_{k=1}^K\norm{x}{H_1}\norm{y}{H_2}\\
&= \sigma_m K\norm{x}{H_1} \norm{y}{H_2}.
\end{align*}
This estimate justifies the rearrangements (since the series are
absolutely convergent) and also yields $\norm{A-F_{m-1}}{}\leq
K\sigma_m$, from which the assertion follows.
\end{proof}

\begin{proposition}
\label{propaddition}
 Suppose that $A_n\in E(a_n,\alpha;H_1,H_2)$ for $1\leq n\leq
  K$. Let 
  $A:=\sum_{n=1}^KA_n$ and $a':=(\sum_{n=1}^Ka_n^{-1/\alpha})^{-\alpha}$. Then 
\begin{itemize}
\item[(i)] $A \in E(a',\alpha)\text{ with } |A|_{a',\alpha}\leq
  K\max_{1\leq n\leq K}|A_n|_{a_n,\alpha}$. In particular 
\[ E(a_1,\alpha)+\cdots + E(a_K,\alpha)\subset E(a',\alpha)\,. \]
\item[(ii)] If both $H_1$ and $H_2$ are infinite-dimensional then
    the 
inclusion above is sharp in the sense that 
\[ E(a_1,\alpha)+\cdots + E(a_K,\alpha)\not\subset E(b,\alpha)\,, \]
whenever $b>a'$. 
\end{itemize}
\end{proposition}
\begin{proof}
Assertion (i) follows from Proposition~\ref{berg:prop10} and 
Corollary~\ref{app:arrangementcoro} (i), which gives 
an upper bound on the rate of decay of the decreasing arrangement of the 
sequences $s(A_1),\ldots,s(A_K)$. 

For the proof of (ii) define $K$ sequences
$s^{(1)},\ldots,s^{(K)}$ by 
\[ s_n^{(k)}:=\exp(-a_kn^\alpha)\quad (n\in\N)\,.\] 
It turns out that it suffices to exhibit $K$ 
compact operators $A_{k}$ with $s_n(A_k)=s_n^{(k)}$ such that 
$s(A)$ is the
decreasing arrangement of the sequences $s^{(1)},\ldots,s^{(K)}$. To
see this, note that then $A_k\in E(a_k,\alpha)$ for
$k\in\set{1,\cdots,K}$. At the same time  
$s_n(A)\geq
\exp(-a'(n+K)^\alpha)$ by Corollary~\ref{app:arrangementcoro} (ii), so
that $A\not\in E(b,\alpha)$ whenever $b>a'$. 

In order to construct these operators we proceed as follows. Since
each $H_i$ 
was assumed to be infinite-dimensional we can choose an orthonormal basis 
$\set{h_n^{(i)}}_{n\in\N}$ for each of them. For each $k=1,\ldots,K$,
we now define a compact operator $A_k:H_1\to H_2$ by  
\[ A_kh_n^{(1)}:=
\begin{cases}
  s_{(n+(k-1))/K}^{(k)}h_n^{(2)} & \text{ for $n\in K\N-(k-1)$},\\
  0 & \text{ for $n\not\in K\N-(k-1)$}.
\end{cases} \]
It is not difficult to see that $s_n(A_k)=s_n^{(k)}$. 
Moreover, it is easily verified that the singular numbers of $A$ are 
precisely the numbers of the form
$s_n^{(k)}$ with $n\in\N$ and $k=1,\ldots,K$. 
Thus, $s(A)$ is the decreasing arrangement of the sequences
$s^{(1)},\ldots,s^{(K)}$ as required. 

\end{proof}

\begin{remark}\label{rem2}
 The proposition implies that $E(a,\alpha)+E(a,\alpha)\subset
 E(2^{-\alpha}a,\alpha)$, but $E(a,\alpha)+E(a,\alpha)\not\subset
 E(a,\alpha)$, because $2^{-\alpha}a<a$. In particular, $E(a,\alpha)$
 is not a linear space.  
\end{remark}

The following result 
establishes a sharp bound on the eigenvalue decay rate
in each exponential class.

\begin{proposition}
\label{prop:weylest}
 Let $a,\alpha>0$ and $A\in E(a,\alpha;H,H)$. Then
\[ \lam(A)\in {\cal E}(a/(1+\alpha),\alpha) \text{ with }
|\lam(A)|_{a/(1+\alpha),\alpha}\leq |A|_{a,\alpha}\,.\]  
If $H$ is infinite-dimensional, 
the result is sharp in the sense that there is an operator $A\in
E(a,\alpha; H,H)$ such that $\lam(A)\not\in {\cal E}(b,\alpha)$ whenever
$b>a/(1+\alpha)$. 
\end{proposition}
\begin{proof}
If $A\in E(a,\alpha)$ then $s_k(A)\leq
  \abs{A}_{a,\alpha} \exp(-ak^\alpha)$.  Using the
  multiplicative Weyl inequality \cite[3.5.1]{Pie} we have
\begin{multline}
\label{multiplicativeweyl}
 \abs{\lam_k(A)}^k\leq\prod_{l=1}^k\abs{\lam_l(A)}\leq
  \prod_{l=1}^ks_l(A)
\leq\prod_{l=1}^k\abs{A}_{a,\alpha} \exp(-al^\alpha)=\\
=\abs{A}_{a,\alpha}^k\exp(-a\sum_{l=1}^k l^\alpha).
\end{multline}

But $\sum_{l=1}^k l^\alpha\geq
\int_0^kx^{\alpha}\,dx=\frac{1}{1+\alpha}k^{\alpha+1}$, which combined
with 
(\ref{multiplicativeweyl}) yields 

\[ |\lam_k(A)|\leq|A|_{a,\alpha}\exp( -ak^\alpha/(1+\alpha)). \]

Sharpness is proved in several steps. We start with the
following
observation. Let $\tau_1\geq \ldots\geq \tau_N\geq 0$ be positive real
numbers. Consider the matrix $C(\tau_1,\ldots,\tau_N)\in L(\C^N,\C^N)$ given
by
\[  C(\tau_1,\ldots,\tau_N):=\left (
\begin{matrix}
0 & \tau_1 & \ldots & 0 & 0 \\
0 & 0 & \ldots & 0 & 0 \\
\vdots & \vdots & & \vdots & \vdots\\
0 & 0 & \ldots & 0 & \tau_{N-1} \\
\tau_{N} & 0 & \ldots & 0 & 0 
\end{matrix}
\right ).
\]

It easy to see that 
\[ s_n(C(\tau_1,\ldots,\tau_N))=\tau_n\]
and 
\[
|\lam_1(C(\tau_1,\ldots,\tau_N))|=\cdots=|
\lam_N(C(\tau_1,\ldots,\tau_N))|=(\tau_1\cdots\tau_N)^{1/N}.
 \]
 
The desired operator is constructed as follows. Fix $a>0$ and
$\alpha>0$. Next choose a
super-exponentially increasing sequence $N_n$, that is, $N_n$ is
increasing and 
$\lim_{n\rightarrow\infty}N_{n-1}/N_n=0$. For definiteness we could
set $N_n=\exp(n^2)$. 

 Put $N_0=0$ and define 
\[d_n:=N_n-N_{n-1} \quad (n\in\N). \]
Define matrices $A_n\in L(\C^{d_n},\C^{d_n})$ by 
\[ A_n=C( \exp(-a(N_{n-1}+1)^\alpha), \ldots,
\exp(-a(N_{n})^\alpha)). \]
Then
\[ s_k(A_n)=\exp( -a(N_{n-1}+k)^\alpha) \quad (1\leq k\leq d_n) \]
and 
\[ |\lam_k(A_n)| =\exp( -ap_n^\alpha) \quad (1\leq k\leq d_n),  \]
where 
\[ p_n:=\frac{1}{d_n}\sum_{l=N_{n-1}+1}^{N_n}l^\alpha. \]

Put 
\[ H:=\bigoplus_{n=1}^\infty \C^{d_n}, \]
and let $A:H\rightarrow H$ be the block-diagonal operator 
\[ (Ax)_n=A_nx_n. \]
Clearly, the singular numbers of $A$ are given by
$s_k(A)=\exp(-ak^\alpha)$ and the 
moduli of the eigenvalues are the numbers $\exp(-ap_n^\alpha)$ occurring with
multiplicity $d_n$. 

Before checking that $A$ has the desired properties we observe that 
\begin{multline}
p_n^\alpha=  \frac{1}{d_n}\sum_{l=N_{n-1}+1}^{N_n}l^\alpha\leq
\frac{1}{d_n}\int_{N_{n-1}+1}^{N_n+1}x^\alpha=\\
=\frac{1}{\alpha+1}\frac{1}{d_n}(
(N_n+1)^{\alpha+1}-(N_{n-1}+1)^{\alpha+1})=
\frac{1}{\alpha+1}N_n^\alpha\delta_n, 
\end{multline}
with $\lim_{n\rightarrow\infty}\delta_n=1$. The latter follows from
the fact that the sequence
$N_n$ was chosen to be super-exponentially increasing. 

Suppose now that $b>a/(\alpha+1)$. Since
$|\lam_{N_n}(A)|=\exp(-ap_n^\alpha)$ we have
\[ |\lam_{N_n}(A)|\exp(bN_n^\alpha)\geq
\exp(-\frac{a}{\alpha+1}N_n^\alpha \delta_n+bN_n^\alpha)=\exp(
N_n^\alpha(b-\frac{a}{\alpha+1}\delta_n)). \]
Thus 
\[ |\lam_{N_n}(A)|\exp(bN_n^\alpha)\rightarrow +\infty \text{ as
}n\rightarrow \infty,\] 
which means that $\lam(A)\not\in {\cal E}(b,\alpha)$.
\end{proof}

\begin{remark}
 Similar results have been obtained by K\"onig and Richter
 \cite[Proposition 1]{KR},
though without estimates on the gauge of $\lambda(A)$.  
\end{remark}

\section{Resolvent estimates}
\label{section:resest}
In this section we shall derive an upper bound for the norm of the resolvent
$(zI-A)^{-1}$ of $A\in E(a,\alpha)$ in terms of the distance of $z$ to
the spectrum of $A$ and the departure from normality of $A$, a 
number quantifying the non-normality of $A$. We shall
employ a technique originally due to Henrici \cite{henrici}, who used
it in a finite-dimensional context. The basic idea is to write $A$ as
a perturbation of a normal operator having the same spectrum as $A$
by a quasi-nilpotent operator. A similar argument can be used to
derive resolvent estimates for operators belonging to Schatten classes
(see \cite{Gil} (and references therein) and
\cite{Ban}). 

Following the idea outlined above we start with bounds for \qn\ operators. 

\begin{proposition}\label{prop3}
Let $a,\alpha>0$. 
\begin{itemize}
\item[(i)] If $A\in E(a,\alpha;H,H)$ 
is \qn, that is, $\sigma(A)=\set{0}$, then 
\begin{equation}\label{qn:est}
\norm{(I-A)^{-1}}{}\leq f_{a,\alpha}(|A|_{a,\alpha}), 
\end{equation}
where $f_{a,\alpha}:\R_0^+\to\R_0^+$ is defined by
\[ f_{a,\alpha}(r)=\prod_{n=1}^\infty(1+r\exp(-an^{\alpha})). \]
Moreover, $f_{a,\alpha}$ has the following asymptoticss:
\begin{equation}\label{prop3:eq3}
 \log f_{a,\alpha}(r) \sim
a^{-1/\alpha}\frac{\alpha}{1+\alpha}(\log r)^{1+1/\alpha} 
\text{ as $r\to\infty$}. 
\end{equation} 
\item[(ii)] If $H$ is infinite-dimensional the 
estimate (\ref{qn:est}) is sharp in the sense that
  there is a \qn\ $B\in E(a,\alpha;H,H)$ such that 
\begin{equation}\label{prop3:sharpest}
\log\norm{(I-zB)^{-1}}{}\sim \log
  f_{a,\alpha}(|zB|_{a,\alpha})\text{ as $|z|\to\infty$}. 
\end{equation}
\end{itemize}
\end{proposition}
\begin{proof} Fix $a,\alpha>0$. 
 
(i) Since $A$ is trace class
  and \qn\ a standard estimate (see, for example, 
\cite[Chapter X, Theorem 1.1]{GGK}) shows that 
\[ \norm{(I-A)^{-1}}{}\leq \prod_{n=1}^\infty (1+s_n(A))\,.\]
Thus
\[  \norm{(I-A)^{-1}}{}\leq
  \prod_{n=1}^\infty (1+|A|_{a,\alpha}\exp(-an^\alpha))=f_{a,\alpha}(|A|_{a,\alpha}). \]
It remains to prove the growth estimate (\ref{prop3:eq3}). We proceed
by noting that $f_{a,\alpha}$ extends to an entire function of genus
zero with $f_{a,\alpha}(0)=1$. 
Moreover, the maximum modulus of $f_{a,\alpha}(z)$ for $|z|=r$ equals
$f_{a,\alpha}(r)$. 
The growth of $f_{a,\alpha}$ can thus be estimated 
by
\begin{equation}
\label{boasjensen} 
N(r) \leq \log f_{a,\alpha}(r)\leq N(r)+Q(r)\,, 
\end{equation} 
where $N(r)=\int_0^rt^{-1}n(t)\,dt$, $Q(r)=r\int_r^\infty
t^{-2}n(t)\,dt$ and $n(r)$ denotes the number of zeros of
$f_{a,\alpha}$ lying in the closed disk with radius $r$ centred at 0
(see \cite[p.~47]{Boa}).  

Since $n(r)=\lfloor a^{-1/\alpha}(\log_+r)^{1/\alpha}\rfloor$, where
$\log_+(r)=\max\set{0,\log r}$ and 
$\lfloor\cdot \rfloor$ denotes the floor-function, 
we have, for $r\geq 1$,  
\begin{align}
N(r)& =a^{-1/\alpha}\int_1^rt^{-1}(\log t)^{1/\alpha}\,dt + O(\log
r)\nonumber \\
& =a^{-1/\alpha}\frac{\alpha}{1+\alpha}(\log r)^{1+1/\alpha} + O(\log
r)\,;\label{prop3:eq:nest}
\end{align}
while $Q$ satisfies 
\begin{equation}\label{prop3:eq:qest}
 Q(r)= O((\log r)^{1/\alpha}) \text{ as
  $r\to\infty$}\,. 
\end{equation}
To see this, note that 
\begin{equation}
\label{qestimate}
Q(r)\leq a^{-1/\alpha}r\int_r^\infty t^{-2}(\log
t)^{1/\alpha}\,dt=a^{-1/\alpha}r\int_{\log r}^\infty e^{-u}u^{1/\alpha}\,du\,;
\end{equation}
putting $r=e^s$ it thus suffices to show that 
\[ e^s\int_s^\infty e^{-u}u^{1/\alpha}\,du=O(s^{1/\alpha})\text{ as
  $s\to\infty$}. \]
This, however,
is the case since
\[
s^{-1/\alpha}e^s\int_s^\infty
e^{-u}u^{1/\alpha}\,du=\int_s^\infty
e^{-(u-s)}(u/s)^{1/\alpha}\,du=\int_0^\infty
e^{-t}(1+t/s)^{1/\alpha}\,dt\to 1\] 
as $s\to\infty$.
Combining (\ref{prop3:eq:qest}), (\ref{prop3:eq:nest}) and  
(\ref{boasjensen})
the growth
estimate (\ref{prop3:eq3}) follows. 

(ii) Since $H$ is infinite-dimensional, we may choose an orthonormal
basis $\set{h_n}_{n\in\N}$. Define
the operator $B\in L(H,H)$ by 
\[ Bh_n:=\exp(-an^\alpha)h_{n+1} \quad (n\in\N). \]
It is not difficult to see that $s_n(B)=\exp(-an^{\alpha})$ for
$n\in\N$, so that $B\in E(a,\alpha;H,H)$. Before we proceed let 
\[ c_n:=\sum_{k=1}^nk^\alpha \quad (n\in\N_0), \] 
and note that since $\int_0^nx^\alpha\,dx\leq \sum_{k=1}^nk^\alpha
\leq \int_1^{n+1}x^\alpha\,dx$, we have
\[ \frac{1}{\alpha+1}n^{\alpha+1}\leq c_n\leq
\frac{1}{\alpha+1}(n+1)^{\alpha+1} \quad (n\in\N_0). \] 
The operator $B$ is \qn, since 
\[ \norm{B^n}{}=\exp(-ac_n)\leq \exp(
-\frac{a}{\alpha+1}n^{\alpha+1})\,,\]
which implies $\norm{B^n}{}^{1/n}\to 0$ as $n\to\infty$.  

It order to determine the asymptotics of $\log
\norm{(I-zB)^{-1}}{}$ we start by noting that 
\begin{multline} \label{prop3:eq10}
\norm{(I-zB)^{-1}}{}^2\geq
\norm{(I-zB)^{-1}h_1}{}^2=\|\sum_{n=0}^\infty
  (zB)^nh_1\|^2\\ =\sum_{n=0}^\infty |z|^{2n}\exp(-2ac_n)\geq
  \sum_{n=0}^\infty
  |z|^{2n}\exp(-2\frac{a}{\alpha+1}(n+1)^{\alpha+1})\geq |z|^{-2} g(|z|),
\end{multline}
where 
\[ g(r):=\sum_{n=1}^\infty
r^{2n}\exp(-2\frac{a}{\alpha+1}n^{\alpha+1}) \quad (r\in\R_0^+). \]
Thus
\begin{equation}\label{prop3:chain}
2\log f_{a,\alpha}(|zB|_{a,\alpha})\geq 2\log
\norm{(I-zB)^{-1}}{}\geq -2\log |z| +\log g(|z|)\,, 
\end{equation}
which shows that in order to obtain the desired asymptotics
(\ref{prop3:sharpest}) 
it suffices to prove that 
\begin{equation}\label{prop3:eq101}
\log g(r) \sim 2 a^{-1/\alpha}\frac{\alpha}{\alpha+1}(\log
r)^{1+1/\alpha}\text{ as $r\to\infty$}.
\end{equation} 
In order to establish the asymptotics above we introduce the maximum term 
\begin{equation}
\label{maximumtermdef}
\mu(r):=\max_{1\leq n<\infty}r^{2n}\exp(
-2\frac{a}{\alpha+1}n^{\alpha+1}) \quad (r\in\R_0^+)\,.
\end{equation}
Since $g$ extends to an entire function of finite order 
we have (see, for example, \cite[Problem 54]{PS})
\[ \log \mu(r)\sim \log g(r)\text{ as $r\to\infty$}\,,\] 
which implies that it now suffices to show that $\mu$ has the desired
asymptotics 
\begin{equation}\label{prop3:muasymp}
\log \mu(r)\sim 2 a^{-1/\alpha}\frac{\alpha}{\alpha+1}(\log
r)^{1+1/\alpha}\text{ as $r\to\infty$}.
\end{equation} 
We now estimate $\mu(r)$ for fixed $r$. 
Define the function $m_r:\R_0^+\to\R_0^+$ by 
\[ m_r(x)=\exp( -2\frac{a}{\alpha+1}x^{\alpha+1}+2x\log r). \]
It turns out that $m_r$ has a maximum 
at $x_r=a^{-1/\alpha}(\log r)^{1/\alpha}$ and that $m_r$ is
monotonically increasing on $(0,x_r)$ and monotonically decreasing on
$(x_r,\infty)$. Thus 
\begin{equation}
\label{maxtermest}
\log m_r(x_r-1)\leq \log \mu(r)\leq \log m_r(x_r)\,.
\end{equation}
Write $x_r-1=\delta_r x_r$ and note that $\delta_r\to 1$ as
$r\to\infty$. Observing that 
\[ \frac{\log m_r(x_r)}{(\log r)^{1+1/\alpha}}=
   2a^{-1/\alpha}\frac{\alpha}{\alpha+1}\]
while 
\[ \frac{\log m_r(x_r-1)}{(\log r)^{1+1/\alpha}}=
\frac{\log m_r(\delta_rx_r)}{(\log r)^{1+1/\alpha}}=
 2a^{-1/\alpha}\left ( -\frac{\delta_r^{\alpha+1}}{\alpha+1}+\delta_r
 \right )
  \to 2 a^{-1/\alpha}\frac{\alpha}{\alpha+1}\]
as $r\to\infty$  
we conclude, using (\ref{maxtermest}), that (\ref{prop3:muasymp})
 holds. This implies (\ref{prop3:eq101}), which yields 
(\ref{prop3:sharpest}), as required. 
\end{proof}
\begin{rembreak}
\label{rem:improvement}
\hspace*{\parindent}
(i) The bound for the growth of the resolvent of a \qn\ 
$A\in E(a,\alpha)$ given in the above proposition is an
improvement compared to those obtainable from the usual   
estimates for operators in the Schatten classes. Indeed, if $A$
belongs to the Schatten $p$-class (i.e.~$s(A)$ is $p$-summable) 
for some $p>0$, then 
\[ \norm{(I-A)^{-1}}{}\leq f_p(\norm{A}{p}) \]
where $\norm{A}{p}$ denotes the Schatten $p$-(quasi) norm of
$A$. Here, $f_p:\R^+_0\to \R_0^+$ is given by 
\[ f_p(r)=\exp(a_pr^p +b_p)\,,\]
where
$a_p$ and $b_p$ are positive numbers depending on $p$, but not on $A$
(see, for example, \cite{Sim1} or \cite[Theorem 2.1]{Ban}, where 
a discussion of the constants $a_p$ and $b_p$ can be found).   

(ii) 
Closer inspection of the proof yields the
 following explicit upper bound for $f_{a,\alpha}$
\[ \log f_{a,\alpha}(r)\leq 
  a^{-1/\alpha}\left ( \frac{\alpha}{1+\alpha} (\log_+r)^{1+1/\alpha} 
   +r\Gamma(1+1/\alpha,\log_+r)
  \right )\,,\]
where 
$\Gamma(\beta,s)$ denotes the incomplete gamma function 
\[ \Gamma(\beta,s)=\int_s^\infty \exp(-t)\,t^{\beta-1}\,dt\,.\] 
This follows  from (\ref{boasjensen}) together with the
estimate $n(r)\leq a^{-1/\alpha} (\log_+r)^{1/\alpha}$. 
\end{rembreak}

A simple consequence of the proposition is the following estimate for
the growth of the resolvent of a \qn\ $A\in E(a,\alpha)$. 
\begin{corollary}
If $A\in E(a,\alpha)$ is \qn, then 
\[ \norm{(zI-A)^{-1}}{}\leq
|z|^{-1}f_{a,\alpha}(|z|^{-1}|A|_{a,\alpha}) \text{ for $z\neq 0$}. 
\]
\end{corollary}

The proposition above can be used to obtain growth estimates for the
resolvents of any $A\in E(a,\alpha)$ by means of the following device.

\begin{theorem}\label{schurdecomp}
Let $A\in S_\infty$. Then  $A$ can be written as a sum 
\[ A=D+N, \]  
such that
\begin{itemize}
\item[(i)] $D\in S_\infty$, $N\in S_\infty$;
\item[(ii)] $D$ is normal and $\lam(D)=\lam(A)$; 
\item[(iii)] $N$ and $(zI-D)^{-1}N$ are \qn\ for every
  $z\in\varrho(D)=\varrho(A)$. 
\end{itemize}
\end{theorem}

\begin{proof} See \cite[Theorem~3.2]{Ban}.  
\end{proof} 

This result motivates the following definition.
\begin{definition}
  Let $A\in \Si$. A decomposition
\[ A=D+N \]
with $D$ and $N$ enjoying properties (i--iii) of the previous theorem
is called a {\it Schur decomposition of $A$}. The operators $D$ and
$N$ will be referred to as the {\it
  normal} and the {\it \qn\ part of the Schur decomposition of $A$}, 
respectively.
\end{definition}
\begin{rembreak}
\hspace*{\parindent}
(i) The terminology stems from the fact that in the finite
  dimensional setting the decomposition in Theorem~\ref{schurdecomp}
  can be obtained as follows: since any matrix is unitarily equivalent to an
  upper-triangular matrix 
by a
  classical result due to Schur, it suffices to establish 
  the result for matrices of this form. In this case, simply choose
  $D$ to be the diagonal part, and $N$ the off-diagonal part of the
  matrix. 

(ii) 
The decomposition is not unique, not even modulo unitary equivalence:
there is a matrix $A$ with two Schur decompositions $A=D_1+N_1$ and
$A=D_2+N_2$ such that $N_1$ is not unitarily equivalent to $N_2$ (see
\cite[Remark 3.5 (i)]{Ban}). 
Note, however, that the normal parts of any two Schur
decompositions of a given compact operator are always unitarily
equivalent. 
\end{rembreak}

Using the results in the previous section we are able to locate the
position of the normal part and the quasi-nilpotent part of an
operator $A\in E(a,\alpha)$ in the scale of exponential classes. 
\begin{proposition}\label{weyl:prop}
  Let $A\in E(a,\alpha)$. If $A=D+N$ is a Schur decomposition of $A$  
with normal part $D$ and \qn\ part $N$, then
\begin{itemize}
\item[(i)] $D \in E(a/(1+\alpha),\alpha)$ with
  $|D|_{a/(1+\alpha),\alpha}\leq |A|_{a,\alpha}$; 
\item[(ii)] $N\in E(a',\alpha)$ with $|N|_{a',\alpha}\leq
  2|A|_{a,\alpha}$, where $a'= a(1+(1+\alpha)^{1/\alpha})^{-\alpha}$.
\end{itemize}
\end{proposition}
\begin{proof}
Since $D$ is normal, its singular numbers 
coincide with its eigenvalues, which in turn coincide with the
eigenvalues of $A$. Assertion (i) is thus a consequence of
Proposition~\ref{prop:weylest}, while assertion (ii) follows from (i) 
and Proposition~\ref{propaddition} by taking $N=A-D$. 
\end{proof}
\begin{remark}
Assertion (i) above is sharp in the following sense: there is $A\in
E(a,\alpha)$ such that for any normal part $D$ of $A$ we have 
$D\not\in E(b,\alpha)$ whenever $b>a/(1+\alpha)$. This follows from
the corresponding statement in Proposition~\ref{prop:weylest} and the fact
that all normal parts of $A$ are unitarily equivalent. 
\end{remark}
For later use we define the following quantities, originally
introduced by Henrici \cite{henrici}.  
\begin{definition}
Let $a,\alpha>0$ and define $\nu_{a,\alpha}:\Si\to
\R_0^+\cup\set{\infty}$ 
by 
\[ \nu_{a,\alpha}(A):=
\inf \all{|N|_{a,\alpha}}{\mbox{$N$ is a \qn\ part of $A$}}\,.\]
We call $\nu_{a,\alpha}(A)$ the {\it $(a,\alpha)$-departure from
  normality of $A$}. 
\end{definition}
\begin{remark} 
Henrici originally introduced this quantity for matrices and 
with the $(a,\alpha)$-gauge of $N$ replaced by the Hilbert-Schmidt
norm. For a discussion of the case where the $(a,\alpha)$-gauge is
replaced by a Schatten norm and its uses to obtain
resolvent estimates for Schatten class operators see \cite{Ban}. 
\end{remark}
The term `departure from normality' is justified in view of the
following characterisation. 
\begin{proposition}
  Let $A\in E(a,\alpha)$. Then 
\[ \nu_{a,\alpha}(A)=0\Leftrightarrow \mbox{ $A$ is normal}. \]
\end{proposition}
\begin{proof}
  Let $\nu_{a,\alpha}(A)=0$. Then there exists a sequence of Schur
  decompositions with \qn\ parts $N_n$ such that $|N_n|_{a,\alpha}\rightarrow
  0$. Thus 
\[ \norm{A-D_n}{}=\norm{N_n}{}=s_1(N_n)\leq
  \exp(-a)|N|_{a,\alpha}\rightarrow 0, \]
where $D_n$ are
  the corresponding normal parts. Thus $A$ is a limit of normal
  operators converging in the uniform operator topology and is
  therefore normal. The converse is trivial.
\end{proof}

For a given $A\in E(a,\alpha)$ the departure from normality  is 
difficult to calculate. 
The following simple but somewhat crude bound is useful in
practice. 
\begin{proposition}\label{npbounds}
Let $A\in E(a,\alpha)$. 
Then 
\[ \nu_{b,\alpha}(A)\leq 2|A|_{a,\alpha}\,\]
whenever $b\leq a(1+(1+\alpha)^{1/\alpha})^{-\alpha}$. 
\end{proposition}
\begin{proof}
  Follows from Proposition~\ref{weyl:prop} together with the fact that
  $|N|_{b,\alpha}\leq |N|_{a',\alpha}$ whenever $b\leq a'$. 
\end{proof}
We are now ready to deduce resolvent estimates for arbitrary $A\in
E(a,\alpha)$. Using a Schur decomposition $A=D+N$ with $D$ normal and
$N$ \qn\ we consider $A$ as a
perturbation of $D$ by $N$. Since $D$ is normal we have  
\begin{equation}\label{res:normal}
\norm{(zI-D)^{-1}}{}=\frac{1}{d(z,\sigma(D))}\quad (z\in\varrho(D)),
\end{equation}
where for $z\in\C$ and $\sigma\subset \C$ closed, 
\[ d(z,\sigma):=\inf_{\lam\in\sigma} |z-\lam| \]
denotes the distance of $z$ to $\sigma$. The influence of the
perturbation $N$ on the other hand, is controlled by
Proposition~\ref{prop3}. All in all, we have the following.

\begin{theorem}\label{carle:theo}
Let $A\in E(a,\alpha)$. If $b\leq 
a(1+(1+\alpha)^{1/\alpha})^{-\alpha}$, then
\begin{equation}
\label{resolventest}
\norm{(zI-A)^{-1}}{}\leq \frac{1}{d(z,\sigma(A))}
f_{b,\alpha} \left ( \frac{\nu_{b,\alpha}(A)}{d(z,\sigma(A))} \right
)\,.
\end{equation}
\end{theorem}
\begin{proof}
  Fix $b\leq a(1+(1+\alpha)^{1/\alpha})^{-\alpha}$. 
Then there is 
  a Schur decomposition of $A$ with normal part $D$ and
  \qn\ part $N\in E(b,\alpha)$ by Proposition~\ref{weyl:prop}. 
Since the bound above is trivial for
  $z\in\sigma(A)$ we may assume $z\in\varrho (A)$. As $D$ and $N$ stem
  from a Schur decomposition of $A$ we see that $(zI-D)^{-1}$ exists
  (because $\sigma(A)=\sigma(D))$ and that $(zI-D)^{-1}N$ is \qn. 
Moreover  
 \[ |(zI-D)^{-1}N|_{b,\alpha}\leq \norm{(zI-D)^{-1}}{}|N|_{b,\alpha}=
  \frac{|N|_{b,\alpha}}{d(z,\sigma(D))}\,,\] 
by (\ref{res:normal}) and Proposition~\ref{prop1}.  
Thus $(I-(zI-D)^{-1}N)$ is invertible in $L$ and 
\[ \norm{(I-(zI-D)^{-1}N)^{-1}}{}\leq 
f_{b,\alpha}\left (\frac{|N|_{b,\alpha}}{d(z,\sigma(D))}\right ), \]
by Proposition~\ref{prop3}. 
Now, since $(zI-A)=(zI-D)(I-(zI-D)^{-1}N)$, we conclude that
$(zI-A)$ is invertible in $L$ and 
\begin{eqnarray*}
\norm{(z-A)^{-1}}{} & \leq & \norm{(I-(zI-D)^{-1}N)^{-1}}{}\norm{(zI-D)^{-1}}{}
\\ 
& \leq & \frac{1}{d(z,\sigma(D))}
f_{b,\alpha}\left (\frac{|N|_{b,\alpha}}{d(z,\sigma(D))}\right ). 
\end{eqnarray*} 
Taking the infimum over all Schur decompositions while using 
$\sigma(A)=\sigma(D)$ once again the result follows. 
\end{proof}
\begin{rembreak}
\hspace*{\parindent}
(i) The estimate remains valid if
    $\nu_{b,\alpha}(A)$ is replaced by something larger, for example by the
    upper bound given in Proposition~\ref{npbounds}. 

(ii) The estimate is sharp in the
  sense that if $A$ is normal then (\ref{resolventest}) reduces
    to the sharp estimate (\ref{res:normal}).
\end{rembreak}
  
\section{Bounds for the spectral distance}
\label{section:specdist}
Using the resolvent estimates obtained in the previous section it is
possible to give upper bounds for the Hausdorff distance of the
spectra of operators in $E(a,\alpha)$. 
Recall that the {\it Hausdorff distance}
${\rm Hdist}\,(.,.)$ is the following 
metric defined on the space of compact subsets of $\C$
\[
{\rm Hdist}\,(\sigma_1,\sigma_2):=\max\set{\hat{d}(\sigma_1,\sigma_2),
  \hat{d}(\sigma_2,\sigma_1)},
\]
where 
\[ \hat{d}(\sigma_1,\sigma_2):=\sup_{\lam\in\sigma_1}d(\lam,\sigma_2) \]
and $\sigma_1$ and $\sigma_2$ are two compact subsets of $\C$. 
For $A,B\in L$ we borrow terminology from matrix perturbation theory
and call $\hat{d}(\sigma(A),\sigma(B))$ the {\it spectral variation of $A$
  with respect to $B$} and $\rm{Hdist}\,(\sigma(A),\sigma(B))$ the
{\it spectral distance of $A$ and $B$}.  

The main tool to bound the spectral variation 
is the following result, which is based on a simple but
powerful argument usually credited to Bauer and Fike~\cite{bauerfike}
who first employed it in a finite-dimensional context. 
\begin{proposition}
  Let $A\in\Si$. Suppose that there is a strictly monotonically
  increasing surjection $g:\R_0^+\to \R_0^+$ such that 
\[ \norm{ (zI-A)^{-1}}{}\leq g( d(z,\sigma(A))^{-1}) 
 \quad (\forall z\in \varrho(A)). \]
Then for any $B\in L$, the spectral variation of $B$ with respect to
$A$ satisfies
\[ \hat{d}(\sigma(B),\sigma(A))\leq h( \norm{A-B}{}), \]
where $h:\R_0^+\to\R_0^+$ is the function defined by 
\[ h(r)=(\tilde g(r^{-1}))^{-1}\,\]
and $\tilde g:\R_0^+\to \R_0^+$ is the inverse of the function $g$. 
\end{proposition}
\begin{proof}
Let $B\in L$. In what follows we shall use the abbreviations
\[ d:=d(z,\sigma(A)),\quad E:=B-A. \]
Without loss of generality we may assume that $E\neq 0$. 
We shall first establish the following implication:
\begin{equation}\label{specdist:imp}
z\in \sigma(B)\cap \varrho (A) \Rightarrow \norm{E}{}^{-1}\leq
\norm{(zI-A)^{-1}}{}.
\end{equation}
To see this let $z\in \sigma(B)\cap \varrho (A)$ and suppose to the contrary that
\[ \norm{(zI-A)^{-1}}{}\norm{E}{}<1. \]
Then $\bigl(I-(zI-A)^{-1}E\bigr)$ is invertible in $L$, 
so $(zI-B)=(zI-A)\bigl(I-(zI-A)^{-1}E\bigr)$ 
is invertible in $L$. Hence 
$z\in\varrho (B)$ which contradicts $z\in\sigma (B)$. Thus the
implication (\ref{specdist:imp}) holds.  

In order to prove the proposition it suffices 
to show that 
\begin{equation}
z\in\sigma (B) \Rightarrow d(z,\sigma(A))\leq h(\norm{E}{})\,,
\label{svar:theo:e1}
\end{equation}
which is proved as follows. Let $z\in\sigma(B)$. 
If $z\in \sigma (A)$ there is nothing
to prove. We may thus assume that $z\in\varrho (A)$. Hence, by (\ref{specdist:imp}),
\[
\norm{E}{}^{-1}\leq\norm{(zI-A)^{-1}}{} \leq g(d^{-1}).  
\]
Since
$g$ is strictly monotonically increasing, so is $\tilde g$. Thus 
\[ \tilde g(\norm{E}{}^{-1})\leq d^{-1}, \]
and hence
\[ d(z,\sigma(A))=d\leq (\tilde g(\norm{E}{}^{-1}))^{-1}=h(\norm{E}{})=h(\norm{A-B}{}).
\]
\end{proof}

Using the proposition above together with the resolvent estimates in
Theorem~\ref{carle:theo} we now obtain the following spectral
variation and spectral distance
formulae.

\begin{theorem} 
\label{theo:specdist}
Let $A\in E(a,\alpha)$ and define 
$a':=a(1+(1+\alpha)^{1/\alpha})^{-\alpha}$. 
\begin{itemize}
\item[(i)] If $B\in L$ and $b>a'$ then 
\begin{equation}
\label{eq:specvar}
\hat d(\sigma(B),\sigma(A))\leq \nu(A)_{b,\alpha} h_{b,\alpha}\left
  ( \frac{ \norm{A-B}{}}{\nu_{b,\alpha}(A)} \right )\,.
\end{equation}
\item[(ii)] If $B\in E(a,\alpha)$ and $b>a'$ then
\begin{equation}
\label{eq:specdist}
{\rm Hdist}\,(\sigma(A),\sigma(B))\leq m h_{b,\alpha}\left
  ( \frac{ \norm{A-B}{}}{m} \right )\,,
\end{equation}
where $m:=\max\set{ \nu(A)_{b,\alpha}, \nu(B)_{b,\alpha}}$. 
\end{itemize}
Here, the
function $h_{b,\alpha}:\R_0^+\to \R_0^+$ is given by 
\[ h_{b,\alpha}(r):=(\tilde g_{b,\alpha}(r^{-1}))^{-1}, \]
where $\tilde g:\R_0^+\to \R_0^+$ is the inverse of
the function $g_{b,\alpha}:\R_0^+\to \R_0^+$ defined by  
\[ g_{b,\alpha}(r):=rf_{b,\alpha}(r)\,, \]
and $f_{a,\alpha}$ is the function defined in Proposition~\ref{prop3}. 
\end{theorem}
\begin{proof} To prove (i) fix $b\geq a'$. The assertion now follows from the
  previous proposition by noting that 
\[ \norm{(zI-A)^{-1}}{}\leq \frac{1}{\nu_{b,\alpha}(A)} g_{b,\alpha}\left (
  \frac{\nu_{b,\alpha}(A)}{d(z,\sigma(A))}\right ) \]
by Theorem~\ref{carle:theo}. 
To prove (ii) fix $b\geq a'$. Then it is not difficult see that 
\[ \norm{(zI-A)^{-1}}{}\leq \frac{1}{m} g_{b,\alpha}\left (
  \frac{m}{d(z,\sigma(A))}\right ),  \]
and 
\[ \norm{(zI-B)^{-1}}{}\leq \frac{1}{m} g_{b,\alpha}\left (
  \frac{m}{d(z,\sigma(B))}\right ),  \]
and the assertion follows as in the proof of (i). 
\end{proof}

\begin{rembreak}
\hspace*{\parindent}
(i) Note that $\lim_{r \downarrow 0}h_{b,\alpha}(r)=0$, so
  the estimate for the spectral distance becomes small when
  $\norm{A-B}{}$ is small. In
  fact, it can be shown that 
\[ \log h_{b,\alpha}(r)\sim
  -b^{1/(1+\alpha)}\left (
   \frac{1+\alpha}{\alpha}\right ) ^{\alpha/(1+\alpha)}|\log
   r|^{\alpha/(1+\alpha)}\text{ as $r\downarrow 0$}. \]
 This follows from the asymptotics in Proposition~\ref{prop3} 
  together with the fact that
   if $\log f(r)\sim a(\log r)^{\beta}$, then $\log \tilde{f}(r)\sim
   a^{-1/\beta}(\log r)^{1/\beta}$ where $\tilde{f}$ is the inverse of
   $f$. 

    (ii) It is not difficult to see, for example by arguing as in
    the proof of part (ii) of the theorem, that the inequalities
    (\ref{eq:specvar}) 
    and (\ref{eq:specdist}) 
    above remain valid if $\nu_{b,\alpha}(A)$ or
    $\nu_{b,\alpha}(B)$ is replaced by
    something larger --- for example, by the upper bounds given in
    Proposition~\ref{npbounds}. 
  
    (iii) Assertion (ii) of the theorem is sharp in the sense
  that if both operators are normal, then (ii) reduces to 
   \[ {\rm Hdist}\,(\sigma(A),\sigma(B))\leq \norml{A-B}\,. \]  
\end{rembreak}

\section{Appendix: Monotone arrangements}
\label{appendixa}
In this appendix we present a number of results concerning sequences 
and their arrangements used in Section~\ref{section:expoclass}. 

Let $a:\N\rightarrow \R$ be a sequence. Define
\[ \norm{a}{+}:=\sup_{n\in\N}a_n, \]
\[ \norm{a}{-}:=\inf_{n\in\N}a_n. \]
For $u:\N\rightarrow \R$ call
\[ {\rm rank}\,(u):={\rm card}\,\all{n\in\N}{u_n\neq 0}. \]
\begin{definition}
  Let $a:\N\rightarrow\R$ be a sequence. Let extended real-valued
  sequences $a^{(+)}$ and $a^{(-)}$ be defined by
\[ a^{(+)}:\N\rightarrow \R\cup\set{-\infty,\infty}, \quad
a^{(+)}_n:=\inf\all{\norm{a-u}{+}}{\text{rank}\,u<n}, \]
\[ a^{(-)}:\N\rightarrow \R\cup\set{-\infty,\infty}, \quad
a^{(-)}_n:=\sup\all{\norm{a-u}{-}}{\text{rank}\,u<n}. \] We call
$a^{(+)}$ the {\it decreasing arrangement of $a$}, and $a^{(-)}$ the
{\it increasing arrangement of $a$}.
\end{definition}
This terminology is justified in view of the fact that $a^{(+)}$
(respectively $a^{(-)}$) is a decreasing (respectively increasing)
sequence.
Moreover, if $a$ is monotonically decreasing, then
$a^{(+)}=a$, and similarly for monotonically increasing sequences. 

More generally, we will consider monotone arrangements of collections
of sequences by first amalgamating them into one sequence and then
regarding the resulting monotone arrangement. A more precise definition
is the following. 
\begin{definition}
Given $K$ real-valued sequences 
$\set{a_n^{(1)}}_{n\in\N},\ldots,\set{a_n^{(K)}}_{n\in\N}$, define a
new sequence $a:\N\to\R$ by 
\[ a_{(k-1)K+i}=a^{(i)}_k \quad \text{ for $k\in\N$ and
  $1\leq i\leq K$}\,.\]
We then call $a^{(+)}$ ($a^{(-)}$) the \textit{decreasing (increasing) 
arrangement of the $K$ sequences  $a^{(1)},\cdots,a^{(K)}$.}
\end{definition}

Our application of decreasing arrangements will typically be to
singular number sequences, all of which converge to zero at some
stretched exponential rate.  Technically and notationally it is
preferable to work with the logarithms of reciprocals of such
sequences, that is, increasing sequences converging to $+\infty$ at
some polynomial rate.

The following is the main result of this appendix.

\begin{proposition}
  Let $\alpha>0$ and $K\in \N$. Suppose that for each $k\in
  \set{1,\ldots,K}$ we are given a real sequence $a^{(k)}$, a
  positive constant $a_k>0$, and a real number $A_k$. Let $a^{(-)}$
  denote the increasing arrangement of the $K$ sequences
  $a^{(1)},\ldots, a^{(K)}$, and
  define
\[ c=\left ( \sum_{k=1}^Ka_k^{-1/\alpha} \right
)^{-\alpha}\,.\]
\begin{itemize}
\item[(i)] If 
\[ a_n^{(k)}\geq a_kn^\alpha +A_k \quad( \forall n\in \N,\,k\in
\set{1,\ldots,K})\,,\]
then 
\[ a^{(-)}_n\geq cn^\alpha + \min\set{A_1,\ldots,A_K}\quad (\forall n\in\N)\,.\] 
\item[(ii)] If 
\[ a_n^{(k)}\leq a_kn^\alpha +A_k \quad( \forall n\in \N,\,k\in
\set{1,\ldots,K})\,,\]
then 
\[ a^{(-)}_n\leq c(n+K)^\alpha + \max\set{A_1,\ldots,A_K}\quad (\forall n\in\N)\,.\]
\end{itemize}
\end{proposition}

\begin{proof}
For $k\in\set{1,\ldots,K}$ set $a_0^{(k)}=-\infty$ and, for $r\in\R$,  
define the counting functions
\begin{align*} 
\mu_k(r)&:={\rm card}\,\all{n\in\N}{a_n^{(k)}\leq r} \\
\mu(r)&:={\rm card}\,\all{n\in\N}{a^{(-)}_n\leq r}\,.
\end{align*}
The following relations are easily verified. We have 
\begin{equation}
\label{ap:countingfuncrel}
\mu(r)=\sum_{k=1}^K\mu_k(r)\,,
\end{equation}
and 
\begin{align}
  a^{(-)}_{\mu(r)}&\leq r \quad (\forall r\in\R) \label{app:lem3:eq1}\,,\\
  \mu(a^{(-)}_n)&\geq n \quad (\forall n\in\N_0) \label{app:lem3:eq2}\,.
\end{align}
(i) Set $C=\min\set{A_1,\ldots ,A_K}$. Since 
for each $k\in\{1,\ldots,K\}$,
\[ \all{n\in\N}{a_n^{(k)}\leq r}\subset
\all{n\in\N}{a_{k}n^\alpha+C\leq r}\,,\]
we have, for $r\geq C$,  
\[ \mu_k(r)\leq {\rm card}\,\all{n\in\N}{a_{k}n^\alpha+C\leq r}=\left
  \lfloor \left ( \frac{r-C}{a_k} \right )^{1/\alpha} \right
  \rfloor\,,\]
where $\lfloor \cdot \rfloor$ denotes the floor function. Thus, using 
  (\ref{ap:countingfuncrel}), we have
\begin{equation}
\label{cyprus} 
\mu(r)=\sum_{k=1}^K\mu_k(r)
\leq (r-C)^{1/\alpha}\sum_{k=1}^K a_k^{-1/\alpha}\,.
\end{equation}
If $n\in\N$ then $a^{(-)}_n\ge C$, so combining 
(\ref{cyprus}) with (\ref{app:lem3:eq2}) gives
\[(a^{(-)}_n-C)^{1/\alpha}\sum_{k=1}^Ka_k^{-1/\alpha}\geq \mu(a^{(-)}_n)\geq n\,,\]
from which (i) follows. 

(ii)  Set $C=\max\set{A_1,\ldots ,A_K}$. Since 
for each $k\in\{1,\ldots,K\}$,
\[ \all{n\in\N}{a_kn^\alpha +C \leq r}\subset
\all{n\in\N}{a_n^{(k)}\leq r}\,,\]
we have, for $r\geq C$,  
\[ \left
  \lfloor \left ( \frac{r-C}{a_k} \right )^{1/\alpha} \right
  \rfloor = {\rm card}\,\all{n\in\N}{a_{k}n^\alpha+C\leq r}\leq 
\mu_k(r)\,.\]
Thus, using (\ref{ap:countingfuncrel}), we have
\begin{equation}
\label{ap:nur}
\mu(r)=\sum_{k=1}^K\mu_k(r)
\geq \left ((r-C)^{1/\alpha}\sum_{k=1}^K a_k^{-1/\alpha}\right )-K\,.
\end{equation}
Now fix $n_0\in\N$. Choose $r_0\geq C$ such that 
\begin{equation}
\label{ap:rsuchthatn}
n_0=\left ((r_0-C)^{1/\alpha}\sum_{k=1}^K a_k^{-1/\alpha}\right )-K\,.
\end{equation}
From (\ref{ap:nur}) and (\ref{ap:rsuchthatn}) we see that
$n_0\le \mu(r_0)$.
Using (\ref{app:lem3:eq1}),
together with the fact
that $n\mapsto a^{(-)}_n$ is monotonically increasing, now yields  
\[ a^{(-)}_{n_0}\leq a^{(-)}_{\mu(r_0)}\leq r_0=c(n_0+K)^\alpha +C\,.\]
Since $n_0$ was arbitrary, (ii) follows. 
\end{proof}
\begin{corollary}
\label{app:arrangementcoro}
  Let $\alpha>0$ and $K\in \N$. Suppose that for each $k\in
  \set{1,\ldots,K}$ we are given a real sequence
  $b^{(k)}$, a positive constant $a_k>0$, and a real
  number $B_k$. Let $b^{(+)}$ denote the decreasing arrangement
  of the $K$ sequences 
$b^{(1)},\ldots,b^{(K)}$ and define 
\[ c:=\left ( \sum_{k=1}^Ka_k^{-1/\alpha} \right
)^{-\alpha}\,.\] 
\begin{itemize}
\item[(i)] If 
\[ b_n^{(k)}\leq B_k \exp(-a_kn^\alpha) \quad( \forall n\in \N,\,k\in
\set{1,\ldots,K})\,,\]
then 
\[ b^{(+)}_n\leq B\exp(-cn^\alpha) \quad (\forall n\in\N)\,,\]
where $B=\max\set{B_1,\ldots,B_K}$.  
\item[(ii)] If 
\[ b_n^{(k)}\geq B_k \exp(-a_kn^\alpha) \quad( \forall n\in \N,\,k\in
\set{1,\ldots,K})\,,\]
then 
\[ b^{(+)}_n\geq B\exp(-c(n+K)^\alpha) \quad (\forall n\in\N)\,,\]
where $B=\min\set{B_1,\ldots,B_K}$.  
\end{itemize}
\end{corollary}

\subsection*{Acknowledgements} 
I would like to thank Oliver Jenkinson for the many 
discussions that helped to shape this article, in particular the
results in Section~\ref{section:expoclass} and in the Appendix. 
The research described in this article was
partly supported by EPSRC Grant GR/R64650/01.


\begin{thebibliography}{OOOO}

\bibitem[Ban]{Ban} OF Bandtlow (2004)  Estimates
for norms of resolvents and an application to the perturbation of
spectra; {\it Math.\ Nachr.} \textbf{267}, 3--11 
\bibitem[BanC]{BC} OF Bandtlow and C-H Chu (2008) Eigenvalue decay of
  operators on harmonic function spaces; \textit{preprint}
\bibitem[BanJ1]{rates} OF Bandtlow and O Jenkinson (2007) 
Explicit a priori bounds
on transfer operator eigenvalues; \textit{Commun.\ Math.\ Phys.}
\textbf{276}, 901--905
\bibitem[BanJ2]{ruelle} OF Bandtlow and O Jenkinson (2007)  
On the Ruelle eigenvalue sequence; to appear in  
\textit{Ergod.\ Th.\ \& Dynam.\ Sys.} 
\bibitem[BanJ3]{decay} OF Bandtlow and O Jenkinson (2008) 
Explicit eigenvalue estimates for transfer operators acting
  on spaces of holomorphic functions; \textit{Adv.\ Math.}
  \textbf{218}, 902--925
\bibitem[BauF]{bauerfike} FL Bauer and CT Fike (1960) Norms and exclusion
  theorems; {\it Num.\ Math.} {\bf 2}, 42--53
\bibitem[BerF]{b} G Bery and G Forst (1975) \textit{Potential Theory on
Locally Compact Groups}; Heidelberg, Springer
\bibitem[Boa]{Boa} RP Boas (1954) {\it Entire Functions}; New York,
  Academic Press
\bibitem[Cal]{calkin} JW Calkin (1941) Two sided ideals and congruences
  in the ring of bounded operators in Hilbert space; \textit{Ann.\
  Math.} \textbf{42}, 839--873  
\bibitem[DS]{DS2} N Dunford, JT Schwartz (1963) {\it Linear
    Operators, Vol.\ 2}; New York, Interscience 
\bibitem[Gil]{Gil} MI Gil' (2003) {\it
   Operator Functions and Localization of Spectra};
 Berlin, Springer
\bibitem[GGK]{GGK} I Gohberg, S Goldberg, MA Kaashoek (1990) {\it
  Classes of Linear Operators Vol.\ 1}; Basel, Birkh\"{a}user
\bibitem[GK]{GK} I Gohberg, MG Krein (1969) {\it Introduction to the
    Theory of Linear Non-Selfadjoint Operators};
Providence, AMS 
\bibitem[Hen]{henrici} P Henrici (1962) Bounds for iterates, inverses,
  spectral variation and fields of values of non-normal matrices; {\it
  Num.\ Math.} {\bf 4}, 24--40
\bibitem[KR]{KR} H K\"onig and S Richter (1984) Eigenvalues of
  integral operators defined by analytic kernels; {\it Math.\ Nachr.}
  {\bf 119}, 141--155
\bibitem[Nel]{Nel} E Nelimarkka (1982) On $\lam(P,N)$-nuclearity and
  operator ideals; {\it Math.\ Nachr.} {\bf 99}, 231--237
\bibitem[Pie]{Pie} A Pietsch (1988) {\it 
  Eigenvalues and $s$-Numbers}; Cambridge, CUP 
\bibitem[Pok]{Pok} A Pokrzywa (1985) On continuity of spectra in norm
  ideals; {\it Lin.\ Alg.\ Appl.} {\bf 69}, 121--130
\bibitem[PS]{PS} G P\'{o}lya and G Szeg\"{o} (1976) {\it Problems and
    Theorems in Analysis, Volume 2}; Berlin, Springer
\bibitem[Rin]{Rin} JR Ringrose (1971) {\it Compact
    Non-Self-Adjoint Operators}; London, van Nostrand 
\bibitem[Rue]{ruellebook} D Ruelle (2004) {\it Thermodynamic
    formalism: the mathematical structures of equilibrium statistical
    mechanics}; Cambridge CUP
\bibitem[Sim1]{Sim1} B Simon (1977) Notes on infinite determinants of Hilbert
  space operators; {\it Adv.\ in Math.}, {\bf 24}, 244--273
\bibitem[Sim2]{simon} B Simon (1979) \textit{Trace ideals and their 
applications}; Cambridge, CUP
\end{thebibliography}
\end{document}